\def\seq#1_#2{\langle #1_#2:#2\in\omega\rangle}
\def\fc#1|#2{#1\uparrow#2}
\def\ain{\subseteq^*}
\def\cl#1{\overline{#1}}

\def\K{{\cal K}}

\def\C{{\mathfrak c}}

\def\set#1:#2.{{\{\,#1: #2\,\}}}

\def\Z{{\mathbb Z}}

\newif\ifdraft
\drafttrue

\documentclass{article}
\title{Precompact groups and convergence.}
\author{Alexander Shibakov}
\usepackage{amssymb}
\usepackage{amsthm}
\usepackage{amsmath}
\ifdraft\usepackage{showlabels}\fi
\usepackage{enumitem}
\usepackage{scrextend}
\newtheorem{theorem}{Theorem}
\newtheorem{lemma}{Lemma}

\newtheorem{corollary}{Corollary}
\newtheorem{remark}{Remark}
\newtheorem{definition}{Definition}
\newtheorem{question}{Question}
\newlist{countup}{enumerate}{10}
\setlist[countup]{label={\rm(\arabic*)}, ref=(\arabic*)}
\newlist{through}{enumerate}{10}
\setlist[through]{label={\rm(\arabic*)}, ref=(\arabic*)}

\begin{document}
\maketitle

\def\lad{\la{D}}
\def\las{\la{S}}
\def\lasm{\la{S^m}}
\def\lasi{\la{S^i}}
\def\lade{\la{2D}}
\def\lase{\la{2S}}
\def\lasem{\la{2S^m}}
\def\lasei{\la{2S^i}}
\def\lado{\lad\setminus\lade}
\def\laso{\las\setminus\lase}
\def\lasom{\lasm\setminus\lasem}
\def\lasoi{\lasi\setminus\lasei}
\def\laf{\la{F}}
\def\U{{\cal U}}
\def\kw{\mathop{k_\omega}}
\def\Pty{{\bf p}}
\def\B{{\mathbb B}}
\def\UU{{\mathbb U}}
\def\S{{\mathbb S}}
\def\SS{{\cal S}}
\def\limD{*}
\def\oG{\bar G}
\def\leql{\leq_L}
\def\geql{\not\leql}
\def\leqp{\leq_p}
\def\geqp{\geq_p}
\def\hgt{{\rm ht}}
\def\MM#1{${\mathbb M}(#1)$}
\def\kwk{{\kw(\K)}}
\def\la#1{{\langle{#1}\rangle}}
\def\le#1{\la{2{#1}}}
\def\lo#1{{\la{#1}\setminus\le{#1}}}
\def\hd#1#2#3{{\bf h}_{#3}^{#1}({#2})}
\def\hda#1#2#3{\hd{#1}{#2}{\la{#3}}}
\def\hdo#1#2#3{\hd{#1}{#2}{\lo{#3}}}
\def\ha#1#2#3{{^\omega\hd{#1}{#2}{#3}}}
\def\gku{{(G, \K, \U)}}
\def\gks#1{{(G#1, \K#1, \U#1)}}
\def\gkup{{(G', \K', \U')}}
\def\hku#1{{(H#1, \K#1, \U#1)}}
\def\Kw{\K_{\omega_1}}
\def\Uw{\U_{\omega_1}}
\def\Gw{G_{\omega_1}}
\def\mcofl{{\bf m}}
\def\ccl#1{{\bf c}{#1}}
\def\H{{\cal H}}
\def\Fn{\mathop{\rm Fn}}
\def\G{{\mathbb G}}
\begin{abstract}
We consider precompact sequential and Fr\'echet group topologies
and show that some natural constructions of such topologies always
result in metrizable groups answering a question of D.~Dikranjan et
al. We show that it is consistent that all sequential precompact
topologies on countable groups are Fr\'echet (or even metrizable). For
some classes of groups (for example boolean) extra set-theoretic
assumptions may be omitted (although in this case such groups do not
have to be metrizable).

We also build (using $\diamondsuit$) an example of a countably compact
Fr\'echet group that
is not $\alpha_3$ and obtain a counterexample to a conjecture of
D.~Shakhmatov as a corollary.
\end{abstract}
\section{Introduction}
Compact topological groups play a special role throughout mathematics,
in areas ranging from probability theory to harmonic analysis. Locally
compact {\em sequential\/} groups (or even countably tight, see below
for the relevant definitions) are metrizable thus shifting the focus
of study to the convergence properties in groups with weaker
compactness conditions. D.~Shakhmatov in~\cite{Sha1} proposes a
systematic study of sequential and Fr\'echet properties in countably
compact, and, more generally, precompact topological groups (recall
that a topological group $G$ is {\em precompact\/} if it can be
densely embedded as a subgroup of a compact group).

The $\Sigma$-product of $\omega_1$ copies of the two element group
provides a standard example of a countably compact Fr\'echet group
that is not first countable. Note that all countable subgroups of this
exapmple are metrizable. Paper~\cite{HRG1} began the study of
precompact Fr\'echet topologies on Abelian groups. Among the results
proved in~\cite{HRG1} is the consistent existence of a nonmetrizable Fr\'echet
precompact topology on every infinite countable Abelian group.

In this paper we consider precompact groups that are sequential. We
show that in ZFC, some natural `minimal' precompact topologies on
countable groups are sequential if an only if they are metrizable
(Corollary~\ref{ssc.m}), answering a question of D.~Dikranjan et.~al
in~\cite{DGT}. We use the {\em effective topology\/} techniques
pioneered by S.~Todor\v cevic (see~\cite{TU} and~\cite{TU1}) to study
such groups.

In the case of general precompact topologies we show that in some
classes of groups (such as {\em boolean}, i.e.~satisfying the identity
$a+a=0$) every countable precompact sequential group is
Fr\'echet (Corollary~\ref{b.f}). We show that consistently, all 
countable precompact sequential groups may be Fr\'echet
(Theorem~\ref{g.Cohen}, see also Remark~\ref{g.HRG}).

In the final section we construct an example of a countably compact
Fr\'echet boolean group that is in some sense `barely Fr\'echet'. More
precisely, the group does not have the $\alpha_3$ property of
A.~Arkhangel'skii (see below for the definitions) which allows us to
disprove a conjecture of D.~Shakhmatov mentioned in~\cite{Sha2}.

We use standard notation and definitions for set-theoretical and
topological concepts (see~\cite{Ku}). All the definitions and
properties used in this paper related to topological groups can be
found in~\cite{ArTk}. Book~\cite{HM} is a good reference on Peter-Weyl
theorem for compact groups. We only use a corollary of Peter-Weyl
theorem that states that a precompact group $G$ can be embedded as a
subgroup in a product of $U(n)$'s. The only propery of $U(n)$ we use
is that it is a compact metric group.

All spaces are assumed to be (completely) regular. We use
$\cl{A}^\tau$ to denote the closure of $A$ in the topology $\tau$ (the
reference to $\tau$ will be omitted if the topology is clear from the
context). We use $\cdot$ and $1$ to denote the operation and the unit
in an arbitrary group $G$. For an Abelian group these become the
traditional $+$ and $0$, respecively. If $G$ is a group and
$D\subseteq G$ then $\lad$ stands for the subgroup (algebraically)
generated by $D$ (i.e.~the smallest subgroup of $G$ that contains
$D$).

A topological space
$X$ is called {\em Fr\'echet\/} if for any $A\subseteq X$ and any
$x\in\cl{A}$ there exists an $S\subseteq A$ such that $S\to x$ (here
$S\to x$ means $S\ain U$ for every open $U\ni x$ if $S$ is infinite
and $S=\{x\}$ otherwise). More generally, $X$
is called {\em sequential\/} if for any $A\subseteq X$ such that
$\cl{A}\not=A$ there exists an $S\subseteq A$ such that $S\to
x\in\cl{A}\setminus A$. Trivially, every Fr\'echet space is
sequential. A free topological group over a one point compactification
of a countable discrete space provides an easy example of a countable
sequential topological group that is not Fr\'echet.

We use the following intrinsic characterization of precompactness
(sometimes called {\em total boundedness}, see~\cite{ArTk}): a topological group $G$ is
precompact if for any nonempty open $U\subseteq$ there exists a finite
$F\subseteq G$ such that $U\cdot F=G$ (or $F\cdot U=G$).

\section{Small precompact groups}

A topological group $G$ is called $ss$-characterized (see~\cite{DGT})
if there exists an $S\subseteq G$ (called the {\em characterizing
sequence}) such that the topology of $G$ is the finest precompact
topology in which $S\to1$.

Recall that a topology $\tau$ on a countable set $X$ (considered
naturally as a subset of $2^\omega$ in the product topology) is called
{\it analytic\/} if $\tau$ is a continuous image of the space of
the irrational numbers. Papers~\cite{TU}, \cite{TU1}, and~\cite{Sh1}
contain a number of results on sequential spaces with analytic
topologies.

\begin{lemma}\label{ssc.a}
Let $G$ be an $ss$-characterized countable topological group.
Then the topology of $G$ is analytic.
\end{lemma}
\begin{proof}
Assume $G=\omega$ and let $S=\set m_n:n\in\omega.$ be the
characterizing sequence. Let $U(\infty)=\prod_{n\in\omega}U(n)$.
Note that for each $n\in\omega$ the group $U(n)$ is naturally embedded
in $U(\infty)$ as a subgroup. Consider the following set
$$
C_n=\set p:G\to U(n)\subseteq U(\infty):p\hbox{ is an algebraic
homomorphism}.
$$
Note that for every $n\in\omega$ the set $C_n$ is a closed subset of
$\prod_{g\in G}U(\infty)$ in the topology of pointwise
convergence. Given $n,m,k\in\omega$ each set
$$
C_n^{m,k}=\set p\in C_n:d(1,p(m_i))<(k+1)^{-1}\hbox{ for }i>m.
$$
is Borel as are the sets $C_n^k=\cup_{m\in\omega}C_n^{m,k}$,
$C^n=\cap_{k\in\omega}C_n^k$ and $C=\cup_{n\in\omega}C^n$.

Suppose $p\in C$ and $\epsilon>0$. Let $k\in\omega$ be such that
$(k+1)^{-1}<\epsilon$. Let $p\in C^n$ for some $n\in\omega$. Then $p\in
C_n^{m,k}$ for some $m\in\omega$ so $d(1,p(m_i))<(k+1)^{-1}<\epsilon$ for
$i>m$. Thus $S\subseteq^*p^{-1}(B^n_\epsilon)$ for evey $\epsilon>0$
where $B^n_\epsilon=\set x\in U(n):d(1,x)<\epsilon.$. Thus $S\to 1$ in
the precompact topology $\tau$ generated on $G$ by the prebase $\set
p^{-1}(B^n_\epsilon):\epsilon>0, p\in C^n.$. 

Let $\tau'$ be the maximal precompact topology on $G$ in which $S\to
1$. Let $U\subseteq G$ be an open neighborhood of $1$. By Peter-Weyl
theorem there are $\epsilon_1,\ldots,\epsilon_l>0$ and continuous
homomorphisms $p_1:G\to U(n_1),\ldots,p_l:G\to U(n_l)$ such that
$\cap_{i\leq l}p_i^{-1}(B^{n_i}_{\epsilon_i})\subseteq U$. Since $S\to
1$ in $\tau'$ and $p_i$ is continuous, for every $k\in\omega$ there is
an $m\in\omega$ such that $d(1,p_i(m_j))<(k+1)^{-1}$ for all $j>m$. So
$p_i\in C_{n_i}^{m,k}\subseteq C_{n_i}^k$ for every $k\in\omega$. Thus $p_i\in
C^{n_i}\subseteq C$ and $\tau'=\tau$.

Let $k\in\omega$ and consider the map $\pi_{k,n}:U(n)^\omega\to2^\omega$
given by
$\pi_{k,n}(x_1,\ldots,x_i,\ldots)=(\sigma_1,\ldots,\sigma_i,\ldots)$ where
$\sigma_i=1$ if $d(x_i,1)<(k+1)^{-1}$ and $\sigma_i=0$ otherwise. Now
$\pi_{k,n}$ is a measurable map. Put $\pi_{k,n}(p)=f^p_k\in2^\omega$ where $p\in
C^n$ and $f^p_k(m)=1$ if and only if $d(1,p(m))<(k+1)^{-1}$. Thus the set
$B=\cup_{k,n\in\omega}\pi_{k,n}(C^n)\subseteq2^\omega$ is an analytic set of
characteristic functions of some prebase at $1\in G$. Thus $\tau$ is
analytic by \cite{TU}, Proposition~3.2(iii).
\end{proof}

The following definition is used in several constructions.

\begin{definition}\label{kw.topology}
Let $(X,\tau)$ be a topological space. Then $(X,\tau)$ is called $\kw$
if there exists a countable family $\K$ of subspaces of $X$ such that
$U\in\tau$ if and only if $U\cap K$ is relatively open in $K$ for
every $K\in\K$. We say that $\tau$ is {\em determined} by $\K$ and
write $\tau=\kw(\K)$.
\end{definition}

The next simple lemma shows that precompactness and $k_\omega$
are in some sense orthogonal properties. For a class of countable
abelian groups whose topology is determined by a $T$-sequence this
result was proved in \cite{PZ}, Proposition~2.3.12.

\begin{lemma}\label{k.precom}
A $k_\omega$ topological group is precompact if and
only if it is compact.
\end{lemma}
\begin{proof}
Let $\K=\set K_n:n\in\omega.$ be a countable family of compact
subspaces of $G$ that determines its topology. We may assume
$1\in K_n\subseteq K_{n+1}$ for every $n\in\omega$. If $G$ is not
compact there exists an infinite closed discrete subset $D=\set
d_n:n\in\omega.\subseteq G$ where $d_n\not=d_m$ for $n\not=m$.

By induction build a sequence $\set O_n:n\in\omega.$ of subsets of $G$
such that
\begin{countup}[series=through]
\item\label{O.open} $1\in O_n\subseteq K_n$, $O_n$ is relatively open
in $K_n$

\item\label{O.up} $\cl{O_n}\subseteq O_{n+1}$ for every $n\in\omega$;

\item\label{O.shift} $(\cl{O_n}\cdot K_m)\cap D=(\cl{O_m}\cdot K_m)$ for every
$m\leq n$.

\end{countup}

If $O_n$ has been built, note that $\cl{O_n}$ is compact so for every
$m\leq n$ there exists an open $O_n^m$ such that $\cl{O_n}\subseteq
O_n^m$ and $(\cl{O_n^m}\cdot K_m)\cap D=(\cl{O_n}\cdot K_m)\cap D$. Put
$O_{n+1}=\cap_{m\leq n}O_n^m\cap
K_{n+1}$. Properties \ref{O.open}--\ref{O.shift} are easy to check.

It follows from the choice of $\K$ and properties~\ref{O.open}
and~\ref{O.up} that $O=\cup_{n\in\omega}O_n$ is an open neighborhood
of $1$ in $G$. Now~\ref{O.shift} and the choice of $D$ imply that
$(O\cdot K_n)\cap D=(\cl{O_n}\cdot K_n)\cap D$ is finite for every $n\in\omega$.

Now if $F\subseteq G$ is finite there is an $n\in\omega$ such that
$F\subseteq K_n$. Thus $O\cdot F\not=G$ for any finite $F\subseteq G$
contradicting the precompactness of $G$.
\end{proof}

The lemmas above can now be used to answer Question~4.6 of~\cite{DGT}.

\begin{corollary}\label{ssc.m}
Every sequential countable subgroup of a countable $ss$-characterized
group is metrizable.
\end{corollary}
\begin{proof}
Being sequential and analytic by~Lemma~\ref{ssc.a} $G$ is either first countable
or $k_\omega$ by \cite{Sh1}, Theorem~1. Since an infinite $k_\omega$ countable group cannot be
precompact by Lemma~\ref{k.precom}, $G$ is
metrizable.
\end{proof}

We now turn to the general (i.e.~not necessarily definable) precompact
countable sequential groups. Recall that a base of open neighborhoods
of $1$ of a topological group $G$ is called {\it linear\/} if every
element of the base is a normal subgroup of $G$.

\begin{lemma}\label{pl.f}
Every countable sequential group with a linear base is
Fr\'echet.
\end{lemma}
\begin{proof}
Let $S\to1$ be an arbitrary convergent sequence. Let $U\subseteq G$ be
an open neighborhood of $1$ and let $\overline{V}\subseteq U$ for some
open subgroup $V\subseteq G$. Then $\overline{\langle S\setminus
F\rangle}\subseteq\overline{V}\subseteq U$ for some finite $F\subseteq
S$. Since $\langle S\setminus F\rangle$ does not contain any isolated
points the set $\overline{\langle S\setminus F\rangle}$ cannot be
countably compact and thus contains an infinite closed and discrete
subset $D_F\subseteq\overline{\langle S\setminus F\rangle}$. Thus $G$
contains a closed copy of the space
$L=\{\omega\}\cup\omega\times\omega$ in which the neighborhoods of the
only nonisolated point $\omega$ are given by
$U_n=\{\omega\}\cup\omega\times(\omega\setminus n)$.

If $G$ is sequential and not Fr\'echet it contains a closed copy of
the sequential fan $S(\omega)$. This is impossible by Lemma~4
in~\cite{BZ}.
\end{proof}

Since every torsion (i.e.~such that for every $g\in G$, $g^n=1$ for some
$n\in\Z$) compact group 
has a
linear base of open neighborhods of $1$ (see, for example,~\cite{HR})
we obtain the following
corollary.

\begin{corollary}\label{b.f}
Every countable sequential precompact group of finite exponent
(i.e.~such that there is an $n\in\Z$ with the property that $g^n=1$
for every $g\in G$, in
particular, every boolean) is Fr\'echet.
\end{corollary}

We do not know if the finite exponent restriction above can be dropped
in ZFC. To prove the consistency of the property above for the class
of all precompact sequential groups we need the following lemma.

\begin{lemma}\label{gen.seq.g}
Let $G$ be a countable sequential precompact group. Let $\H$ be a
countable family of nowhere dense subsets of $G$. Then there exists an
infinite $S\subseteq G$ such that $S\to1$ and $S\cap H$ is finite for
every $H\in\H$.
\end{lemma}
\begin{proof}
Let $\H=\set H_i:i\in\omega.$ be a family of nowhere dense subsets of $G$
with the property that for every convergent sequence $S\subseteq G$
such that $S\to1$ there is an $H_i\in\H$ such that $H_i\cap S$ is
infinite. By extending $\H$ if necessary we may assume that $1\in
H_0$, each $H_i$
is closed in $G$ and $H_i\cdot F\in\H$ for every $H_i\in\H$ and every
finite $F\subseteq G$. Let $\cl{G}\supseteq G$ be a compact group that
contains $G$ as a dense subgroup. Then for each $H_i\in\H$ its closure
$\cl{H_i}$ in $\cl{G}$ is nowhere dense and for any $S\subseteq G$
such that $S\to g\in G$ there exists an $i\in\omega$ such that
$H_i\cap S$ is infinite.

Using recursion, build a countable set $D=\set
d_i:i\in\omega.\subseteq G$ and a
family $\U=\set U_i:i\in\omega.$ of open subsets of $\cl{G}$ such that
$\cl{U_{i+1}}\subseteq U_i$,
$\cl{U_i}\cap(\cup_{j\leq i}\cl{H_j})=\varnothing$, and $d_k\in
U_i\setminus\set d_j:j<k.$
for every $i\in\omega$ and $k>i$.

Note that the set $D$ has the property that for every $i\in\omega$
there exists an open set $U\subseteq\cl{G}$ such that $U\cap D$ is
finite and $\cup_{j\leq i}\cl{H_j}\subseteq U$.

Let $\set F_i:i\in\omega.$ list all the finite subsets of $G$. Build,
by induction, subsets $V_i\subseteq\cup_{j\leq i}\cl{H_j}$ so that the
following properties are satisfied.

\begin{countup}[through]
\item\label{V.open.tower}$1\in V_0$, each $V_i$ is a relatively open subset of
$\cup_{j\leq i}\cl{H_j}$ and $\cl{V_i}\subseteq V_{i+1}$ for every
$i\in\omega$;

\item\label{V.shifts}$(\cl{V_i}\cdot F_j)\cap D=(\cl{V_j}\cdot F_j)\cap D$ for
any $j\leq i$

\end{countup}

Put $V_0=\cl{H_0}$.
Suppose $V_0$, $\ldots$, $V_i$ have been built. Note that
$\cl{V_i}\subseteq\cl{G}$ is compact and
$(\cl{V_i}\cdot \cup_{j\leq i}F_j)\subseteq\cup_{j<I}\cl{H_j}$ for some $I\in\omega$
by~\ref{V.open.tower}, and the properties of $\H$. By the property of
$D$, there exists a subset $U\supseteq(\cl{V_i}\cdot \cup_{j\leq i}F_j)$, open in
$\cl{G}$ such that $U\cap D$ is finite. Using the compactness of
$\cl{V_i}$, find an open subset $U'\supseteq\cl{V_i}$ of $\cl{G}$
such that $(\cl{U'}\cdot \cup_{j\leq i}F_j)\subseteq U$ so
$F^j=(\cl{U'}\cdot F_j)\cap D$ is finite for each $j\leq i$.

Using a similar argument and~\ref{V.shifts} for each $j\leq i$ find an open
$U_j\supseteq\cl{V_i}$ such that $(\cl{U_j}\cdot F_j)\cap
F^j=(\cl{V_i}\cdot F_j)\cap F^j=(\cl{V_i}\cdot F_j)\cap D=(\cl{V_j}\cdot F_j)\cap D$. 
Put $V_{i+1}=\bigcup_{j\leq i+1}\cl{H_j}\cap\bigcap_{j\leq i}U_j\cap
U'$.

Define $V=\cup_{i\in\omega}V_i\cap G$. If $F\subseteq G$ is finite
then $F=F_i$ for some $i\in\omega$ and by~\ref{V.open.tower}
and~\ref{V.shifts} the set $(V\cdot F)\cap D=(\cl{V_i}\cdot F_i)\cap D$ is finite.

Suppose $G\setminus V$ is not closed. Since $G$ is sequential, there
exists an $S\subseteq G\setminus V$ such that $S\to g\in V$ for some
$g$. By the choice of $\H$ there exist $i,j\in\omega$ such that
$g\in V_i$, $i\geq j$, and $H_j\cap S$ is infinite. Since $V_i$ is
relatively open in $\cup_{k\leq i}\cl{H_k}\supseteq H_j$, $V_i\cap S$ is
infinite, contradicting $S\cap V=\varnothing$.

Thus $V$ is a non empty open subset of $G$ such that $G\not=V\cdot F$ for
any finite $F\subseteq G$ contradicting the precompactness of $G$.
\end{proof}

Lemma~\ref{gen.seq.g} can be used to show the consistency of the
property in Lemma~\ref{b.f} in the class of general precompact
countable sequential groups.

\begin{theorem}\label{g.Cohen}
In the model obtained by adding $\omega_2$ Cohen reals to a model of
CH every countable precompact sequential group is Fr\'echet.
\end{theorem}
\begin{proof}
The proof is almost a word-for-word reproduction of the proof of
Lemma~7 of~\cite{Sh3} so we only give a short outline. Let $M$ be an
$\omega$-closed ($[M]^{\omega}\subseteq M$) elementary submodel of
$H(2^{\C^+})$ of size $\omega_1$ such that $G,\dot\tau\in M$ for some
$\Fn(\omega_1,2)$-name 
$\dot\tau$ of a sequential precompact topology on $G$. Let the
ground model $V$ satisfy CH
and let $\G$ be a $\Fn(\omega_1,2)$-generic set over $V$. Just as in
the proof of Lemma~7
of~\cite{Sh3}, one shows that the group
$(G,\tau_M)\in V[\G\cap M]$ is
sequential, where $\tau_M$ consists of $\G\cap M$-interpretations of
sets in $\dot\tau(\alpha)$ where $\alpha\in M$ (here $\tau$ is
interpreted as an $\omega_2$-list of subsets of $G$). Using
elementarity, one shows that $(G,\tau_M)$ is also precompact.

The rest of the proof is nearly identical to that of Lemma~7
in~\cite{Sh3} and uses Lemma~\ref{gen.seq.g} instead of Lemma~6
of~\cite{Sh3}.
\end{proof}

\begin{remark}\label{g.HRG}
\rm
Lemma~\ref{gen.seq.g} may also be used to complement Lemma~16
of~\cite{Sh5} in the proof of Theorem~2 of~\cite{Sh5} to obtain a
model of ZFC in which every countable precompact sequential group is
metrizable thus generalizing Lemma~\ref{ssc.m} to the class of
countable precompact sequential groups.
The saturation argument that uses the elementary submodel $M$
in the theorem above is then replaced by the use of
$\diamondsuit(S_1^2)$ in the ground model and the
$\sigma$-centeredness of the forcing. We omit the details.
\end{remark}

\section{A Fr\'echet-Urysohn group.}
Recall the definition of $\alpha_i$ properties introduced by
A.~Arkhangel'skii under a different name. Several of these properites,
including a number of variations, have been independently defined and
studied by other authors under different names although
Arkhangle'skii was likely the first to undertake a systematic study of
spaces with these properties. We follow the established notation below.

\begin{definition}[\cite{Ar1}, \cite{Ar2}, see also~\cite{Sha1}]\label{alphai}
Let $X$ be a topological space. For
$i=1$, $2$, $3$, and $4$ we say that $X$ is an {\em $\alpha_i$-space\/}
provided for every countable family $\set S_n:n\in\omega.$ of
sequences converging to some point $x\in X$ there exists a `diagonal'
sequence $S$ converging to $x$ such that:
\begin{itemize}
\item[\rm($\alpha_1$)]$S_n\setminus S$ is finite for all $n\in\omega$,

\item[\rm($\alpha_2$)]$S_n\cap S$ is infinite for all $n\in\omega$,

\item[\rm($\alpha_3$)]$S_n\cap S$ is infinite for infinitely many
$n\in\omega$,

\item[\rm($\alpha_4$)]$S_n\cap S\not=\varnothing$ for infinitely many
$n\in\omega$.

\end{itemize}
\end{definition}

P.~Nyikos (see~\cite{Ny1}) noted that $\alpha_i$ properties are of
great utility in the study of convergence properties in groups and
other algebraic objects. Thus he showed in~\cite{Ny1} that every
Fr\'echet group is $\alpha_4$. A number of authors have since
uncovered various connections (or lack thereof) between $\alpha_i$
properties in the presence of an algebraic structure and other
restrictions. Survey~\cite{Sha1} provides a fairly
comprehensive overview of these results.

In this section we build an example of a countably compact Fr\'echet
boolean group that is not $\alpha_3$. Such a group always has a base
of open neighborhoods of $0$ consisting of subgroups (this follows
from Pontryagin's duality, although in the example below this property
can be verified directly) which allows us to answer a question asked
by D.~Shakhmatov in 1990.

We borrow some terminology and techniques from \cite{ShaShi} for the
analysis of topologies on the boolean group.

\begin{definition}\label{kpair}
Call $(G, \K)$ a {\em $\kw$-pair (with respect to $\tau$)\/}
if $(G,\tau)$ is a boolean topological
group with the $\kw$ topology $\tau$ and $\K$ is a countable family of
compact subspaces of $G$ closed under finite sums and intersections
such that $\tau=\kw(\K)$ and $\cup\K=G$.
\end{definition}

\begin{lemma}[\cite{ShaShi}]\label{free.sequence}
Let $(G, \K)$ be a $\kw$-pair, $D'\subseteq G$ be
infinite, closed and discrete in $\kw(\K)$.
Then there exists an
infinite independent $D\subseteq D'$ such that $\lad$ is closed and
discrete in $\kwk$.
\end{lemma}

The topology of the example will be a simultaneous limit of $\kw$ and
first countable precompact topologies. The next definition is a
convenient shortcut and can be viewed as an approximation of the final
topology.

\begin{definition}\label{contrip}
Call $\gku$ a {\em convenient triple\/} if $G$ is a boolean
group, $\U$ is a countable family of subgroups closed under finite
intersections that forms an open base of neighborhoods of $0$ in some
Hausdorff precompact topology $\tau(\U)$ on $G$, and $\K$ is a
countable family of compact (in $\tau(\U)$) subgroups of $G$ closed under finite sums
and intersections such that $\cup\K=G$.
\end{definition}

The next lemma was proved in~\cite{ShaShi}.

\begin{lemma}[\cite{ShaShi}]\label{ct.resolve}
Let $\gku$ be a convenient triple, let $H$ be a subgroup of $G$ closed in
$\kwk$. Then there exists a countable family of open (in $\kw(\K)$)
subgroups of finite index $\U_0\supseteq\U$ such that
$\cl{H}^{\tau(\U_0)}\cap G=\cap\set U\in\U_0:H\subseteq U.=H$.
\end{lemma}

The construction of the example is by an $\omega_1$ recursion. Each
step consists of adding a new convergent sequence such that the
following properties are satisfied.

\begin{definition}\label{pse}
Let $\gku$ and $\gks'$ be convenient triples and an
independent $D\subseteq G$ be such that $\lad$ is closed and discrete
in $\kwk$. Call $\gks'$ a {\em primitive sequential
extension (pse for short) of $\gku$ over $D$\/} if the
following conditions hold:

\begin{countup}[through]
\item\label{pse.order}
$G\subseteq G'$, $\cl{U}^{\tau(\U')}\in\U'$ for every $U\in\U$, and
$\K'$ is the closure of $\K\cup\{L\}$ under finite sums where
$L=\cl{\lad}^{\tau(\U')}$ (thus $L$ is compact in $\kw(\K')$);

\item\label{pse.resolve}
$\cl{\la{D\setminus F}}^{\tau(\U')}\cap G=\la{D\setminus
F}$ for any $F$;

\end{countup}
\end{definition}

It is easy to show that in the notation above, $K'=K+L$ for every
$K'\in\K'$ and some $K\in\K$.
The proof of the next lemma can be extracted from the results
of~\cite{ShaShi} but we present it here for the reader's convenience.

\begin{lemma}\label{pse.successor}
Let $\gku$ be a convenient triple and let $D\subseteq G$ be a countable
infinite independent subset such that $\lad$ is closed and discrete in $\kwk$. Then
there exists a pse $\gks'$ of $\gku$ over $D$. If $D\to0$ in
$\tau(\U)$ then $\gks'$ can be chosen so that $D\to0$ in $\tau(\U')$.
\end{lemma}
\begin{proof}
Let $F\subseteq D$ be finite. Use Lemma~\ref{ct.resolve} to find a
countable family $\U_F$ of open in $\kwk$ subgroups of $G$ of finite
index such that $\la{D\setminus F}=\cap\set U\in \U_F:\la{D\setminus
F}\subseteq U.$. Let $\U''=\U\cup\set
U\in \bigcup_{F\in[D]^{<\omega}}\U_F: D\ain U.$ and let $\U'''$ be the
closure of $\U''$ under finite intersections. Since every element of
$\U'''$ is a subgroup of finite index of $G$ open in $\kwk$ and $\gku$
is a convenient triple, $\tau(\U''')$ is precompact. If $D\to0$ in
$\tau(\U)$ then $D\to0$ in $\tau(\U''')$ by the choice of $\U'''$.

Let $(G'',\tau)$ be the compact group such that $G\subseteq G''$ is a
dense subgroup and $\tau|_G=\tau(\U''')$. Let $L=\cl{\lad}^\tau$ and
let $\K'$ be the closure of $\K\cup\{L\}$ under finite sums. Put
$G'=G+L\subseteq G''$ and $\U'=\set \cl{U}^{\tau|_{G'}}:U\in\U'''.$.

Let $F\subseteq G$ and $g\in G$ be such that $g\not\in\la{D\setminus
F}$. If $g\not\in\lad$ let $U\in\U_\varnothing$ be such that
$D\subseteq U$ and $g\not\in U$. Then $U\in\U'''$ is clopen in $\kwk$
so $g\not\in\cl{U}^{\tau(\U')}$.

If $g\in\lad$ then $g\not\in\la{D\setminus\{f\}}$ for some $f\in F\cap
D$. Let $U\in\U_{\{f\}}$ be such that $D\setminus\{f\}\subseteq U$ and
$g\not\in U$. The rest of the argument is similar to the case when
$g\not\in\lad$. Thus~\ref{pse.resolve} holds. The rest of the
properties follow from the construction.
\end{proof}

We now turn to the limit stages of the construction.
Let $\gamma$ be an ordinal. Suppose for every $\sigma<\gamma$ a
convenient triple $\hku{^\sigma}$ is defined so that the following
conditions hold:

\begin{countup}[through]
\item\label{stack.order}
$H^{\sigma'}\subseteq H^\sigma$, $\K^{\sigma'}\subseteq \K^\sigma$,
and $\U^{\sigma'}\subseteq \set U\cap H^{\sigma'}:U\in\U^\sigma.$
if $\sigma'\leq\sigma<\gamma$;

\item\label{stack.dense}
$H^{\sigma'}$ is dense in $H^\sigma$ in $\kw(\K^\sigma)$ for every
$\sigma'\leq\sigma$;

\end{countup}
Define $\hku{^{<\gamma}}$ by taking
$H^{<\gamma}=\cup_{\sigma<\gamma}H^\sigma$,
$\K^{<\gamma}=\cup_{\sigma<\gamma}\K^\sigma$,
$\U^{<\gamma}=\set \cl{U}^{\kw(\K^{<\gamma})}:U\in\U^\sigma, \sigma<\gamma.$.

Note that in the case of a successor $\gamma$,
$\hku{^{<\gamma}}=\hku{^{\gamma'}}$ where $\gamma'+1=\gamma$.

\begin{lemma}[\cite{ShaShi}]\label{stack.limit}
The family $\U^{<\gamma}$ forms a base of clopen subgroups of finite
index for a precompact group
topology $\tau(\U^{<\gamma})$ on $H^{<\gamma}$ and each $H^\sigma$, $\sigma<\gamma$ is dense
in $H^{<\gamma}$ in $\kw(\K^{<\gamma})$. If $\gamma<\omega_1$ then
$\hku{^{<\gamma}}$ is a convenient triple.
\end{lemma}

If, in addition to properties~\ref{stack.order} and~\ref{stack.dense},
each $\hku{^\sigma}$ is a pse of $\hku{^{<\sigma}}$ over some
$D\subseteq H^{<\sigma}$, we will call
$\set\hku{^{\sigma}}:\sigma<\gamma.$ a {\em pse-chain}.

To ensure that the final group does not satisfy $\alpha_3$ the
following stronger property is introduced. Below $K$ will be a element
of $\K$ for some convenient triple $\gku$ and $\set S_n:n\in\omega.$
will be a specially chosen family of convergent sequences in $G$.

\begin{countup}[through]
\item\label{g.position}
let $K\subseteq G$, then there exists an $n(K)\in\omega$ such that
$K\cap\langle S_n\rangle$ is finite for every $n>n(K)$.

\end{countup}

The preservation of~\ref{g.position} after adding a new convergent
sequence is the subject of the next lemma.

\begin{lemma}\label{ortho}
Let $(G, \K)$ be a $\kw$-pair such that each $K\in\K$
satisfies~\ref{g.position}. Suppose also that each $S_n\subseteq K(n)$
for some $K(n)\in\K$. Let $D\subseteq G$ be an infinite set such
that $\lad$ is closed and discrete in $G$. Then
for every $K\in\K$ the set $K+\lad$ satisfies~\ref{g.position}.
\end{lemma}
\begin{proof}
Let $K\in{\cal K}$ and $n(K)\in\omega$ be such that $K\cap \langle
S_n\rangle$ is finite for every $n>n(K)$. Note that
$\overline{K+\langle S_n\rangle}$ is 
compact in $\kw(\K)$ so $F=\langle D\rangle\cap (K+\la{S_n})$ is
finite for every $n\in\omega$. If $n>n(K)$ the set $F'=K\cap\la{S_n}$
is finite. Suppose $d\in\lad$, $s\in\la{S_n}$, and $a\in K$ are such
that $a+d=s$. Then $d\in F=\set a^i+s^i:a_i\in K, s^i\in\la{S_n},
i\in|F|.$ so $s=a+a^i+s^i$ for some $i\in|F|$ and $a+a^i\in F'=\set
f^j:j\in|F'|.$. Thus $s=f^j+s^i$ for some $i\in|F|$, $j\in|F'|$. Hence
$K+\lad\cap\la{S_n}$ is finite for every $n>n(K)$.
\end{proof}

The first step of the construction is given in the following simple
lemma.

\begin{lemma}\label{g.kw}
There exists a convenient triple $\gku$ such that there are
infinite countable independent sets $S_n\subseteq G$ where $S_n\to0$
for every $n\in\omega$ and for any $K\in\K$ there exists an
$n\in\omega$ such that $K\cap\la{S_n}=\varnothing$.
\end{lemma}
\begin{proof}
It is straightforward to construct a compact boolean group
$T=\cl{\las}$ where $S\to0$ is an infinite independent subset. Let $G$
be the inductive limit of $T^n$, $n\in\omega$, and $\tau(\U)$ be the
topology inherited from the product topology on $T^\omega$. Put
$S_n=\{0\}^{n-1}\times S$. Note that the topology of $G$ is determined
by the family $\K=\set T^n:n\in\omega.$. All the properties of $\gku$
are easy to check.
\end{proof}

We now present the recursive construction of the main example.
Since all the groups $G_\alpha$ in the construction have cardinality
$2^\omega$ we will assume that every $G_\alpha$ is a subgroup
(algebraically) of $2^\omega$. Let $\set
C_\alpha:\alpha<\omega_1.\subseteq[2^\omega]^\omega$ and $\set
P_\alpha:\alpha<\omega_1.\subseteq[2^\omega]^\omega$ be some families
of subsets of $2^\omega$. 

Let $\gks{_0}=\gku$ where the convenient triple $\gku$ and the
sets $S_n\subseteq G$ have been constructed in Lemma~\ref{g.kw}.
Below we use the notation
$\ccl{C_\alpha}$ for the closure of $C_\alpha$ in $\kw(\K_{<\alpha})$.

\begin{lemma}\label{g.ind}
There exists a pse-chain $\set\gks{_\alpha}:\alpha<\omega_1.$ such that
\begin{countup}[through]
\item\label{g.fr}every $K\in\K_\alpha$ satisfies~\ref{g.position};

\item\label{g.cc}
if $P_\alpha\subseteq G_{<\alpha}$ is an infinite subset that is
closed and discrete in $\kw(\K_{<\alpha})$ then
there exists an infinite $S_\alpha\subseteq P_\alpha$ such that
$S_\alpha\to s_\alpha$ in $\kw(\K_{\alpha)})$;

\item\label{g.so}
if $0\in\overline{P_\alpha}^{\kw(\K_{<\alpha})}$ then there
exists an infinite $S_\alpha\subseteq P_\alpha$ such that
$S_\alpha\to0$ in $\kw(\K_{\alpha)})$; 

\item\label{g.seq}
if $0\in\cl{\ccl{C_\alpha}}^{\tau(\U_{<\alpha})}$ then there exists
an infinite $T_\alpha\subseteq C_\alpha$ such that
$T_\alpha\to0$ in $\kw(\K_{\alpha})$;

\end{countup}
\end{lemma}
\begin{proof}
Let $D'\subseteq G_{<\alpha}$ be such that
$0\in\cl{D'}^{\tau(\U_{<\alpha})}$ and there is no $S\subseteq D'$
such that $S\to0$ in $\kw(\K_{<\alpha})$. Pick an infinite $D\subseteq
D'$ such that $D\to0$ in $\tau(\U_{<\alpha})$. Then $D$ is closed and
discrete in $\kw(\K_{<\alpha})$ so using Lemma~\ref{free.sequence} we
may assume that $\lad$ is closed and discrete in
$\kw(\K_{<\alpha})$. Apply Lemma~\ref{pse.successor} to construct a
pse $\gks{_\alpha}$ of $\gks{_{<\alpha}}$ over $D$ such that $D\to0$ in
$\kw(\K_\alpha)$. 

Using either $P_\alpha$ of $\ccl{C_\alpha}$ in place of $D'$, the
argument in the previous paragraph, and putting $S_\alpha=D$ one can
show that~\ref{g.so} and~\ref{g.seq} are
satisfied. Property~\ref{g.cc} can be treated using a similar
construction (note that the conditions for~\ref{g.cc} and~\ref{g.so}
are mutually exclusive) by omitting the condition that $D\to0$ in
$\tau(\U_{<\alpha})$.

Since $\lad$ was chosen to be closed and discrete in
$\kw(\K_{<\alpha})$, Lemma~\ref{ortho} implies that $K+\lad$
satisfies~\ref{g.position}. If $K'\in\kw(\K_\alpha)$ then
$K'=K+\cl{\lad}^{\tau(\U_\alpha)}$ for some
$K\in\kw(\K_{<\alpha})$. Suppose $g\in K'\cap S_n$ for some
$n\in\omega$. Then $g=a+d$ for some $a\in K$ and
$d\in\cl{\lad}^{\tau(\U_\alpha)}$. Since $g\in G_{<\alpha}$, $d\in
G_{<\alpha}$ so by~\ref{pse.resolve} $d\in\lad$. Thus $K'$
satisfies~\ref{g.position} and~\ref{g.fr} holds.
\end{proof}

We next show that when the families $C_\alpha$ and $P_\alpha$ are
chosen to have some special properties, the construction of
Lemma~\ref{g.ind} results in a desired group.

Suppose $\diamondsuit$ holds and let $\set C_\alpha:\alpha<\omega_1.$ be a
$\diamondsuit$-sequence. Identifying $\omega_1$ and $2^\omega$ we may assume that
each $C_\alpha\subseteq2^\omega$.
Let $\set P_\alpha:\alpha<\omega_1.$ list all infinite countable subsets of
$2^\omega$ so that each $P_\alpha$ is listed $\omega_1$ times.

\begin{lemma}\label{e}
Let $\gks{_{\omega_1}}=\gks{_{<\omega_1}}$ where $\gks{_\alpha}$,
$\alpha<\omega_1$ have been constructed in Lemma~\ref{g.ind}. Then
$\kw(\Kw)=\tau(\Uw)$, $\Gw$ is countably compact, Fr\'echet, and not
$\alpha_3$.
\end{lemma}
\begin{proof}
Suppose $A\subseteq\Gw$ is such that $0\not\in A$, $A\cap K$ is closed
for every $K\in\Kw$ and $A\cap U\not=\varnothing$ for every
$U\in\Uw$.

Let $\theta$ be large enough. Consider the sets of the form $A\cap M$
where $M$ is a countable elementary submodel of
$H(\theta)$ and $X\in M$ is a countable set containing the details of the
construction of $\Gw$. The set
$$
\set\gamma\in\omega_1:\gamma=M\cap\omega_1, X\in M, M\leq H(\theta).
$$
is a club in $\omega_1$. Thus $C_\gamma=A\cap M$ for some
$\gamma<\omega_1$ where $M\cap\omega_1=\gamma$. Note that
$\gamma<\omega_1$ is a limit and $\ccl{C_\gamma}=A\cap G_{<\gamma}$.

Let $U\in\U_\alpha$ for some $\alpha<\gamma$. Then
$\cl{U}^{\kw(\Kw)}\in\tau(\Uw)$ so there exists a $\beta\geq\alpha$, a
$K\in\K_\beta$ and a $g\in A\cap K\cap\cl{U}^{\kw(\K_\beta)}$ by the
choice of $A$. By elementarity, we may assume that $\beta<\gamma$ and
$g\in M$ so $g\in A\cap\cl{U}^{\kw(\K_{<\gamma})}$. Thus
$0\in\cl{\ccl{C_\gamma}}^{\tau(\U_{<\gamma})}$ and by~\ref{g.seq}
there exists a $T_\gamma\subseteq\ccl{C_\gamma}\subseteq A$ such that
$T_\gamma\to0$ in $\kw(\K_\gamma)$. This shows that
$\kw(\Kw)=\tau(\Uw)$.

Now simple arguments show that~\ref{g.cc} implies the countable
compactness of $\kw(\Kw)$, \ref{g.so} implies that $\kw(\Kw)$ is
Fr\'echet, and~\ref{g.fr} implies that $\kw(\Kw)$ is not $\alpha_3$.
\end{proof}

Lemmas~\ref{g.ind} and~\ref{e} now imply the following theorem.

\begin{theorem}[$\diamondsuit$]\label{f.na3}
There exists a countably compact boolean Fr\'echet group that is not
$\alpha_3$.
\end{theorem}

The next corollary provides a counterexample to Conjecture~9.4
in~\cite{Sha2} and a partial negative answer to Question~4.3 in~\cite{Sha1}.

\begin{corollary}[$\diamondsuit$]\label{c.f.na3}
There exists a precompact countable Fr\'echet boolean group $G$ with a base of
neighborhoods of $0$ consisting of subgroups of finite index such that
$G$ is not $\alpha_3$.
\end{corollary}
\begin{proof}
Let $G$ be the countable group algebraically generated by the sheaf
that provides a counterexample to the $\alpha_3$ property in the group
from Theorem~\ref{f.na3}.
\end{proof}

\section{Questions.}
Corollary~\ref{b.f} and Theorem~\ref{g.Cohen} (see also the remark
following Theorem~\ref{g.Cohen}) leave open a number of natural
questions. The question below seems to be open even for $\Z$.

\begin{question}\label{cps.f}
Is every countable precompact sequential group Fr\'echet in ZFC?
\end{question}

A negative answer to Question~\ref{cps.f} would follow from a
negative answer to Question~\ref{Sw.p} below (in ZFC). Such a negative
answer would also allow the (consistent) extension of the conclusion of
Theorem~\ref{g.Cohen} to the class of all (not necessarily countable)
precompact groups.

\begin{question}\label{Sw.p}
Can a precompact sequential group contain a closed copy of $S(\omega)$?
\end{question}

The conclusion of Lemma~\ref{ssc.a} can be generalized to groups
having the maximal precompact topology in which a countable family of
sequences converge to $1$. A more careful proof shows that when
`countable' is replaced by `analytic' the resulting topology can be
shown to be {\em coanalytic\/} (i.e.~a complement of some analytic
set). Note also that it can be shown that the set of convergent
sequences of a space with an analytic topology is coanalytic. This
leads to the following natural question.

\begin{question}\label{mpg.a}
Let $G$ be a countable group and ${\cal S}\subseteq 2^G$ be a Borel
(analytic, coanalytic) family of subsets. Let $\tau$ be the maximal
precompact group topology on $G$ such that $S\to1$ in
$\tau$ for each $S\in{\cal S}$. Must $\tau$ be analytic?
\end{question}

The result of Theorem~\ref{f.na3} hints at the possible positive
answers to the next two questions. Note that such groups do not exist
in ZFC alone and cannot be $\alpha_3$. Thus a positive answer to one
of the questions below would strengthen the conclusion of
Theorem~\ref{f.na3}.

\begin{question}\label{cc.F.p}
Do there exist countably compact Fr\'echet groups $H$ and $G$ such
that $H\times G$ is not Fr\'echet?
\end{question}

\begin{question}\label{cc.F.s}
Does there exist a countably compact Fr\'echet group whose square is
not Fr\'echet?
\end{question}

\end{document}